\documentclass{amsart}\usepackage{latexsym,amssymb,graphicx}

\newtheorem{thm}{Theorem}[section]
\newtheorem{lem}[thm]{Lemma}
\newtheorem{cor}[thm]{Corollary}
\theoremstyle{definition}
\newtheorem{defn}[thm]{Definition}

\newtheorem{rem}[thm]{Remark}

 \newcommand{\R}{\mathbb{R}}
\newcommand{\Ess}{\textsc{Essence}} \newcommand{\Conn}{\textsc{Conn}}
\newcommand{\size}[1]{\ensuremath{\vert #1 \vert}}
\newcommand{\perimeter}{\text{\sf Perimeter}}
\newcommand{\weight}{\text{\sf Weight}}

\begin{document}
\title{Windmills and extreme $2$-cells}
\author[J.~McCammond]{Jon McCammond$\ \!{ }^1$}
      \address{Dept. of Math.\\
               University of California, Santa Barbara\\
               Santa Barbara, CA 93106}
      \email{jon.mccammond@math.ucsb.edu}
\author[D.~Wise]{Daniel Wise$\ \!{ }^2$}
      \address{Dept. of Math.\\
               McGill University \\
               Montreal, Quebec, Canada H3A 2K6 }
      \email{wise@math.mcgill.ca}
\date{\today}

\begin{abstract}
  In this article we prove new results about the existence of
  $2$-cells in disc diagrams which are extreme in the sense that they
  are attached to the rest of the diagram along a small connected
  portion of their boundary cycle.  In particular, we establish
  conditions on a $2$-complex $X$ which imply that all minimal area
  disc diagrams over $X$ with reduced boundary cycles have extreme
  $2$-cells in this sense.  The existence of extreme $2$-cells in disc
  diagrams over these complexes leads to new results on coherence
  using the perimeter-reduction techniques we developed in an earlier
  article.  Recall that a group is called coherent if all of its
  finitely generated subgroups are finitely presented.  We illustrate
  this approach by showing that several classes of one-relator groups,
  small cancellation groups and groups with staggered presentations
  are collections of coherent groups.
\end{abstract}

\maketitle \footnotetext[1]{Partially supported by grants from the
  NSF} \footnotetext[2]{Research supported by grants from NATEQ and
  NSERC.}

In this article we prove some new results about the existence of
extreme $2$-cells in disc diagrams which lead to new results on
coherence.  In particular, we combine the diagram results shown here
with the theorems from \cite{McCammondWiseCoherence} to establish the
coherence of various classes of one-relator groups, small cancellation
groups, and groups with relatively staggered presentations.  The
article is organized as follows: \S~\ref{sec:defs} contains background
definitions, \S~\ref{sec:perimeter} recalls how extreme $2$-cells lead
to perimeter reductions and to coherent fundamental groups,
\S~\ref{sec:windmill} introduces the concept of a windmill,
\S~\ref{sec:extreme} uses windmills to prove that extreme $2$-cells
exist, and finally \S~\ref{sec:apps} uses extreme $2$-cells to prove
that various groups are coherent. For instance, we obtain the
following special case of Corollary~\ref{cor:t-windmills}:

\begin{cor}
  Let $G = \langle a_1, \dots, a_r,t\mid W^N \rangle$ where $W$ has
  the form
  \[t^{\epsilon_1}W_1 t^{\epsilon_2} W_2 \ldots t^{\epsilon_k}W_k,\]
  $N$ is arbitrary, and for each $i$, $\epsilon_i$ is a nonzero
  integer and $W_i$ is a reduced word in the $a_i$.  Suppose that
  $\{W_1, W_2, \dots, W_k\}$ freely generate a subgroup of the free
  group $\langle a_1, \dots, a_r\mid - \rangle$.  Then $G$ is
  coherent.
\end{cor}

\section{Basic Definitions}\label{sec:defs}

In this section we review some basic definitions about $2$-complexes
and diagrams.

\begin{defn}[Combinatorial maps and complexes]\label{def:combinatorial}
  A map $Y\rightarrow X$ between CW complexes is \emph{combinatorial}
  if its restriction to each open cell of $Y$ is a homeomorphism onto
  an open cell of $X$.  A CW complex $X$ is \emph{combinatorial}
  provided that the attaching map of each open cell of $X$ is
  combinatorial for a suitable subdivision.  All complexes and maps
  considered in this article will be combinatorial after suitable
  subdivisions.  In addition, we will only consider $2$-complexes in
  which the attaching maps of $2$-cells are immersions.
\end{defn}

\begin{defn}[Polygon]
  A \emph{polygon} is a $2$-dimensional disc whose cell structure has
  $n$~$0$-cells, $n$~$1$-cells, and one $2$-cell where $n\geq 1$ is a
  natural number.  If $X$ is a combinatorial $2$-complex then for each
  open $2$-cell $C\hookrightarrow X$ there is a polygon $R$, a
  combinatorial map $R\rightarrow X$ and a map $C \rightarrow R$ such
  that the diagram
  \[\begin{array}{ccc}
    C & \hookrightarrow & X\\
    \downarrow &\nearrow  & \\
    R &          &\\
  \end{array}\]
  commutes, and the restriction $\partial R \rightarrow X$ is the
  attaching map of $C$.  In this article the term \emph{$2$-cell} will
  always mean a combinatorial map $R\rightarrow X$ where $R$ is a
  polygon. The corresponding \emph{open $2$-cell} is the image of the
  interior of $R$.

  A similar convention applies to $1$-cells.  Let $e$ denote the graph
  with two $0$-cells and one $1$-cell connecting them.  Since
  combinatorial maps from $e$ to $X$ are in one-to-one correspondence
  with the characteristic maps of $1$-cells of $X$, we will often
  refer to a map $e \rightarrow X$ as a {\em $1$-cell} of $X$.
\end{defn}

Technical difficulties with $2$-complexes often arise because of the
existence of redundant $2$-cells and $2$-cells attached by proper
powers.

\begin{defn}[Redundant $2$-cells]
  Let $X$ be a $2$-complex.  If $R$ and $S$ are distinct $2$-cells in
  $X$ with identical boundary cycles then $R$ and $S$ are called
  \emph{redundant $2$-cells}.  More specifically, there must exist a
  combinatorial map $R \to S$ so that $\partial R \hookrightarrow R
  \to S \to X$ agrees with the map $\partial R \hookrightarrow R \to
  X$.
\end{defn}

\begin{defn}[Exponent of a $2$-cell]\label{def:exponent}
  Let $X$ be a $2$-complex, and let $R \rightarrow X$ be one of its
  $2$-cells.  Let $n$ be the largest number such that the map
  $\partial R \rightarrow X$ can be expressed as a path $W^n$ in $X$,
  where $W$ is a closed path in $X$.  This number $n$, which measures
  the periodicity of the map of $\partial R\rightarrow X$, is the
  \emph{exponent} of $R$, and a path such as $W$ is a \emph{period}
  for $\partial R$.  Notice that any other closed path which
  determines the same cycle as $W$ will also be a period of $\partial
  R$.  If the exponent $n$ is greater than $1$, then the $R$ is said
  to be \emph{attached by a proper power}.
\end{defn}

\begin{defn}[Disc Diagrams]
  A \emph{disc diagram} $D$ is a finite non-empty contractible
  $2$-complex together with a specific embedding of $D$ in $\R^2$.  A
  disc diagram which consists of a single $0$-cell is called
  \emph{trivial}.  If it is homeomorphic to a disc then it is
  \emph{non-singular}.  It is a fundamental result in combinatorial
  group theory that the image of a closed (combinatorial) loop $P\to
  X$ is null-homotopic if and only if there is a disc diagram $D\to X$
  having $P$ as its boundary cycle \cite{LySch77}.
\end{defn}

\begin{defn}[Area]
  Let $X$ be a $2$-complex and let $D\to X$ be a disc diagram.  The
  \emph{area of $D$} is simply the number of $2$-cells it contains.
  Since area is a non-negative integer, for every closed loop $P\to X$
  whose image is null-homotopic, there is a minimal area disc diagram
  $D\to X$ having $P$ as its boundary cycle.
\end{defn}

\begin{defn}[Cancellable pair]\label{def:cancellable}
  Let $X$ be a $2$-complex, let $D\to X$ be a disc diagram and let
  $R_1$ and $R_2$ be distinct $2$-cells in $D$.  If (1) $\partial R_1$
  and $\partial R_2$ are lifts of the same loop in $X$, (2) $\partial
  R_1 \cap \partial R_2$ contains a vertex $v$ and (3) the closed path
  $\partial R_1$ can be read counterclockwise starting at $v$ and the
  closed path $\partial R_2$ can be read clockwise starting at $v$ so
  that they have identical images in $X$, then $R_1$ and $R_2$ are
  called a \emph{cancellable pair}.  The definition of a cancellable
  pair is often restricted to the case where $\partial R_1$ and
  $\partial R_2$ contain a $1$-cell in common, but this restriction is
  actually unnecessary.
\end{defn}

\begin{rem}[Redundant cells and proper powers]
  The focus of Definition~\ref{def:cancellable} is on $\partial R_1$
  and $\partial R_2$ (rather than $R_1$ and $R_2$ themselves) because
  of the possibility of redundant $2$-cells and $2$-cells attached by
  proper powers.  If $R$ and $S$ are redundant $2$-cells and $D\to X$
  is a disc diagram containing a $2$-cell $R'$ which maps to $R$, then
  the map $D\to X$ can be modified so that $R'$ is sent to $S$ while
  keeping the rest of the map fixed.  Similarly, if $R$ is a $2$-cell
  in $X$ with exponent $n$ and $D\to X$ is a disc diagram containing a
  $2$-cell $R'$ which is sent to $R$, then there are $n$ distinct ways
  of sending $R'$ to $R$ while keeping the rest of the map fixed.
  Moreover, these modifications do not fundamentally change the basic
  properties of the disc diagram.
\end{rem}

We will need the following lemma about minimal area diagrams.  Its
proof is standard and will be omitted. The basic idea is that $R_1$
and $R_2$ can be ``cut out'' and the resulting hole can be ``sewn
up'', but there are a few technicalities.  See \cite{Ol91} or
\cite{McWi-fl} for complete details.

\begin{lem}\label{lem:minimal}
  Let $X$ be a $2$-complex and let $D\to X$ be a disc diagram.  If $D$
  contains a cancellable pair then $D$ does not have minimal area.
\end{lem}

\section{Perimeter reductions}\label{sec:perimeter}

As mentioned in the introduction, the main goal of this article is to
use structures we call ``windmills'' (introduced in the next section)
to force disc diagrams to contain extreme $2$-cells.  Once this fact
is known in a particular context, the machinery constructed in
\cite{McCammondWiseCoherence} can be used to conclude that the
corresponding fundamental groups are coherent.  In this short section,
we briefly review the main ideas and results from
\cite{McCammondWiseCoherence} and very briefly explain the connection
between the existence of extreme $2$-cells and coherent fundamental
groups.

Let $Y$ be a subcomplex of a $2$-complex $X$.  The \emph{perimeter} of
$Y$ in $X$ is essentially the length of the boundary of an
$\epsilon$-neighborhood of $Y$ in $X$, under the assumption that the
$1$-cells of $X$ have unit length.  For example, the perimeter of a
single edge $e$ is just the number of sides of $2$-cells of $X$ that
are attached to $e$.  Alternatively, the perimeter of $Y$ in $X$ is
the total number of missing sides, where a side of a $2$-cell in $X$
is \emph{missing} if it is attached to a $1$-cell in $Y$ but it is not
a side of a $2$-cell in $Y$.  There is also a weighted version where
the sides of the $2$-cells of $X$ are given non-negative weights
(subject to minor restrictions).  The \emph{weighted perimeter} of $Y$
in $X$ is then the sum of the weights of the missing sides.

The main idea of \cite{McCammondWiseCoherence} is to use perimeter
calculations to force the termination of the following algorithm.  Let
$X$ be $2$-complex with a finitely generated fundamental group and let
$Y$ be a compact subcomplex of $X$ such that the induced map $\pi_1 Y
\to \pi_1 X$ is onto.  Note that such a $Y$ always exists since we can
use the union of closed loops representing a finite generating set.
At this point the map from $\pi_1 Y$ to $\pi_1 X$ may or may not be
$\pi_1$-injective.  If it is, then $\pi_1 Y = \pi_1 X$ and the
compactness of $Y$ implies that $\pi_1 X$ is finitely presented.  If
this map is not $\pi_1$-injective then it is natural to focus
attention on a closed loop $P \to Y \subset X$ that is essential in
$Y$ and null-homotopic in $X$.  Being null-homotopic in $X$ there is a
disc diagram $D \to X$ with $P$ as its boundary cycle and being
essential in $Y$ there is at least one $2$-cell of $D$ that is not in
$Y$.  If we enlarge $Y$ by adding the $2$-cells from $D$, then this
new complex has a fundamental group that still maps onto $\pi_1 X$, it
is still compact and it is closer to being $\pi_1$-injective.  In
general, this process of enlargement might need to happen infinitely
many times.

If, however, all the disc diagrams over $X$ always have $2$-cells
where most of their boundary cycle is contained in the closed loop
$P$, then it is at least conceivable that we can guarantee the
existence of a $2$-cell in $D$ whose addition to $Y$ results in a
larger subcomplex with a smaller (weighted) perimeter.  Under such
conditions, the iterative procedure described above must stop since at
each stage the non-negative integral perimeter of the resulting
subcomplexes is steadily decreasing, and when it stops, $\pi_1 Y' =
\pi_1 X$ and the compactness of $Y'$ implies that $\pi_1 X$ is
finitely presented as above.

Many variations on this proof-scheme are described in
\cite{McCammondWiseCoherence} along with precise definitions and
statements of the results.  In this article we focus on producing
$2$-complexes for which every disc diagram has an extreme $2$-cell.
The conclusion that the corresponding fundamental groups are coherent
will follow from the fact that in the contexts described weights can
be found so that the hypotheses of Theorem~7.6 of
\cite{McCammondWiseCoherence} are satisfied.

\section{Windmills}\label{sec:windmill}

In this section we introduce a particular type of (weak) subcomplex of
a $2$-complex that we call a windmill.  These structures will be used
to force the existence of extreme $2$-cells in disc diagrams.

\begin{defn}[Subcomplexes]\label{def:subcomplexes}
  Let $X$ and $Y$ be $2$-complexes and let $Y \hookrightarrow X$ be a
  topological embedding.  If $X$ and $Y$ can be subdivided so that
  $Y\hookrightarrow X$ is combinatorial, then we will call $Y$ a
  \emph{subcomplex} of $X$ even though its image is not a subcomplex
  in the original cell structure of $X$.  We will use the term {\em
    true subcomplex} if $Y\subset X$ is a subcomplex in the
  traditional sense - without subdivisions.  Finally, given a
  subcomplex $Y$ in $X$, the closure of $X\setminus Y$ will be another
  subcomplex that we will call its \emph{complement}.
\end{defn}

The fact that the image of $Y$ need not be a subcomplex of $X$ in the
traditional sense could have been avoided if we had assumed at the
start that $X$ and $Y$ were already suitably subdivided.  We will
not, however, carry out such subdivisions since the cell structures of
$X$ and $Y$ carry information of interest in applications.  In fact
we will mostly be interested in the other extreme: subcomplexes where
the image of $Y^1$ is, in some sense, transverse to $X^1$.

\begin{figure}[ht]\centering
  \includegraphics[width=.3\textwidth]{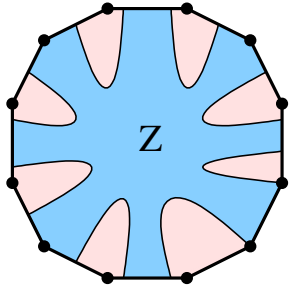}
  \caption{A windmill configuration in a
    $2$-cell.}\label{fig:windmill}
\end{figure}

\begin{defn}[Windmills]\label{def:windmill}
  Let $X$ be a $2$-complex, let $Y \hookrightarrow X$ be a subcomplex
  and let $Z \hookrightarrow X$ be its complement subcomplex, and let
  $\Gamma = Y\cap Z$ be the subgraph of $X$ which separates them.  If
  $\phi:R\to X$ is a $2$-cell of $X$, then we will say $R$ is a
  windmill with respect to $Z$ if, roughly speaking, $\phi^{-1}(Z)$
  looks like a windmill.  An example is shown in
  Figure~\ref{fig:windmill}.  The dark portion of this $2$-cell
  belongs to $\phi^{-1}(Z)$ and there are eight $1$-cells in its
  interior which separate the light and dark areas.  Other examples
  are shown in Figures~\ref{fig:shading} and~\ref{fig:bc}.

  The precise definition we will use goes as follows: $R$ is a
  \emph{windmill with respect to $Z$} if $\phi^{-1}(Z \setminus
  \Gamma)$ is connected and $\phi^{-1}(\Gamma)$ is homeomorphic to a
  collection of isolated points in $\partial R$ plus $n\geq 2$
  disjoint closed $1$-cells whose endpoints lie in $\partial R$ and
  whose interiors lie entirely in the interior of $R$.  If each
  $2$-cell of $X$ is a windmill with respect to $Z$, then $Z$ is a
  \emph{windmill in $X$}.
\end{defn}

Note that if $Z$ is a windmill in $X$ then $\Gamma\cap X^1$ is a
finite set of points.  We will now give two concrete methods of
creating windmills which we will need for our applications in
Section~\ref{sec:apps}.

\begin{defn}[$\partial A$]\label{def:pA}
  Let $X$ be a connected $2$-complex and let $A$ be a true subcomplex
  of $X^1$. Let $Y$ be the closure of a regular neighborhood of $A$,
  let $Z$ denote the complementary subcomplex, and let $\partial A$
  denote the intersection $\Gamma$ of $Y$ and $Z$.  Then $Z$ is a
  windmill if $\phi^{-1}(A)$ is disconnected for each $2$-cell
  $\phi:R\to X$.  For example, if $X$ consists of a single hexagonal
  $2$-cell and $A$ consists of five of its $0$-cells and one of its
  $1$-cells, then the windmill created by $\partial A$ is shown in
  Figure~\ref{fig:shading}.
\end{defn}

\begin{figure}[ht]\centering
  \includegraphics[width=.3\textwidth]{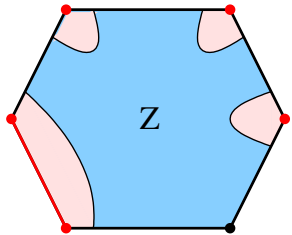}
  \caption{A windmill created by $\partial A$.}\label{fig:shading}
\end{figure}

Our second construction is similar.

\begin{defn}[$\eth A$]\label{def:eA}
  Let $X$ be a connected $2$-complex, let $A$ and $B$ be true
  subcomplexes of $X^1$ such that $(A \cup B) = X^1$ and $(A\cap B)
  \subset X^0$.  We will now define $\eth A$ and simultaneously define
  $Y$ and $Z$ so that they extend $A$ and $B$ respectively.  Define
  the vertices of $\eth A$ to be the $0$-cells in $A\cap B$.  Suppose
  $R\to X$ is a $2$-cell.  If $\partial R\to X$ only contains
  $1$-cells from $A$ then $R$ will also be a $2$-cell of $Y$ and if
  $\partial R$ only contains $1$-cells from $B$ then $R$ will belong
  to $Z$.  Finally, if it contains $1$-cells from $A$ and $B$, then
  the boundary cycle $\partial R\to X$ can be uniquely partitioned
  into non-trivial paths which alternate between paths in $A$ and
  paths in $B$.  For each non-trivial path in $A$, we add an edge to
  $\eth A$ which starts and ends at the endpoints of this path and
  runs parallel to it through the interior of $R$.  The regions of $R$
  thus created which border $1$-cells from $A$ will belong to $Y$ and
  the unique remaining region will belong to $Z$.  This procedure will
  create a windmill $Z$ if $\phi^{-1}(A)$ has more than one
  non-trivial component for each $2$-cell $\phi:R\to X$.  If $X$
  consists of a single hexagonal $2$-cell, $A$ contains three of its
  $1$-cells (the two leftmost $1$-cells and the $1$-cell in the upper
  right) and $B$ contains the other three, then the windmill created
  by $\eth A$ is shown in Figure~\ref{fig:bc}.
\end{defn}

\begin{figure}[ht]\centering
  \includegraphics[width=.3\textwidth]{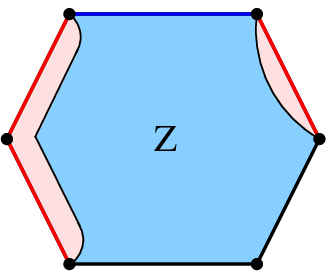}
  \caption{A windmill created by $\eth A$.}\label{fig:bc}
\end{figure}

\begin{rem}[$\partial A$ versus $\eth A$]
  Despite their similar definitions, in general, neither $\partial
  A\hookrightarrow X$ nor $\eth A\hookrightarrow X$ is homotopic to a
  subgraph of the other.  To pass from $\partial A\hookrightarrow X$
  to $\eth A\hookrightarrow X$ requires shrinking some ``trivial''
  loops and identifying distinct vertices.  Moreover, these
  definitions will lead to independent applications.
\end{rem}

The windmills of primary interest will be those where a particular
inclusion map is $\pi_1$-injective.  When $X$ has no redundant
$2$-cells and no $2$-cells attached by proper powers (and $\Gamma$ is
the subgraph which separates a windmill from its complement), we will
require that the inclusion $\Gamma\hookrightarrow X$ be
$\pi_1$-injective.  In the general case, we will only need to focus on
a particular portion of $\Gamma$ that we call its essence.

\begin{defn}[Essence of a subgraph]
  Let $X$ be a $2$-complex and let $\Gamma\hookrightarrow X$ be the
  subgraph which separates a windmill in $X$ from its complement.  If
  $\Gamma$ partitions redundant $2$-cells $R$ and $S$ in similar ways,
  then the portion of $\Gamma$ in $R$ and the portion in $S$ perform
  similar functions in disc diagrams over $X$ and we will not need
  both.  Similarly, if $R$ is a $2$-cell with exponent $n>1$ and the
  windmill-like structure in $R$ respects this $n$-fold symmetry, then
  we will only need ``$\frac1n$-th'' of $\Gamma \cap R$.  Both types
  of redundancies may occur in $\partial A$ and in $\eth A$.  These
  two observations define an equivalence relation on the $1$-cells of
  $\Gamma$.  Let $\Ess(\Gamma)\hookrightarrow X$ be a graph in $X$
  which results from picking one $1$-cell from each equivalence class.
  In the end the exact choice of $1$-cells is irrelevant since, if
  $\Gamma'\hookrightarrow X$ and $\Gamma''\hookrightarrow X$ are any
  two possibilities for $\Ess(\Gamma)\hookrightarrow X$ then $\Gamma'$
  and $\Gamma''$ are homeomorphic and and the maps are homotopic.  To
  see the homotopy, note that distinct choices of representative
  $1$-cells can be pushed to the same path in $X^1$ while keeping
  their endpoints fixed.
\end{defn}

\begin{defn}[Splitting windmills]\label{def:splitting}
  Let $X$ be a $2$-complex and let $\Gamma$ be the subgraph which
  separates a windmill $Z$ in $X$ from its complement.  If the
  embedding $\Ess(\Gamma)\hookrightarrow X$ is $\pi_1$-injective on
  each connected component then $Z$ is a \emph{splitting windmill}.
\end{defn}

\section{Extreme $2$-Cells}\label{sec:extreme}

In this section we prove that minimal area disc diagrams over
$2$-complexes with splitting windmills have $2$-cells which are
extreme in the sense that they are attached to the rest of $D$ along a
very small portion of their boundary cycle
(Theorem~\ref{thm:extreme}).  The key property of splitting windmills
that enables the proof is that they partition minimal area disc
diagrams in a very restrictive manner.  Given any map to $X$ (such as
a disc diagram) we can pull back the partitioning of $X$ determined by
the windmill and its complement to define a partitioning of the
domain.  Recall that a graph with no cycles is a \emph{forest}, a
connected forest is a \emph{tree}, and a vertex of valence~$1$ is a
\emph{leaf}.

\begin{thm}[Forest]\label{thm:forest}
  Let $X$ be a $2$-complex and let $\Gamma\hookrightarrow X$ be the
  subgraph which separates a splitting windmill from its complement.
  If $P\to X$ is a non-trivial null-homotopic immersed combinatorial
  path and $\psi:D\to X$ is a minimal area disc diagram having $P$ as
  its boundary cycle, then $\Gamma' = \psi^{-1}(\Gamma)$ is a forest
  and every leaf in $\Gamma'$ lies in $\partial D$.
\end{thm}

\begin{proof}
  Let $Y$ and $Z$ denote complementary subcomplexes in $X$, one of
  which is a splitting windmill.  Which letter represents the windmill
  will be irrelevant since the proof is symmetric with respect to $Y$
  and $Z$.  The second assertion is immediate since every $1$-cell in
  $\Gamma'$ traverses an open $2$-cell of $D$ in which one side
  belongs to $\phi^{-1}(Y)$ and the other to $\phi^{-1}(Z)$, whereas
  if $D$ contained a leaf in its interior, both sides of its unique
  $1$-cell would necessarily belong to the same preimage.

  Suppose that $\Gamma'$ contains a cycle.  By choosing an innermost
  cycle we can find a cycle $Q$ in $\Gamma'\subset D$ so that the
  portion of $D$ to the left of $Q$ belongs entirely to $\psi^{-1}(Y)$
  or entirely to $\psi^{-1}(Z)$ as $Q$ is traversed counterclockwise.
  Without loss of generality assume it belongs to $\psi^{-1}(Z)$.  If
  $\psi(Q)$ is not an immersed loop in $\Gamma$, then the $2$-cells
  containing the portion of $Q$ immediately before and after a point
  which fails to be an immersion will form a cancellable pair in $D$.
  Note that we need the fact that the portion of $D$ to the left of
  $Q$ lies in $\psi^{-1}(Z)$ to conclude that these $2$-cells have
  opposite orientations.  Since by Lemma~\ref{lem:minimal} this
  contradicts our assumption that $D$ has minimal area, $\psi(Q)$ must
  be immersed.  Moreover, since $\Gamma$ is a graph, $\psi(Q)$ is an
  essential in $\Gamma$.

  Next, let $\phi:\Gamma\to \Ess(\Gamma)$ be the natural projection
  which sends each $1$-cell in $\Gamma$ to the $1$-cell in
  $\Ess(\Gamma)$ which represents its equivalence class.  We claim
  that $\phi(\psi(Q))$ is immersed -- hence essential -- in
  $\Ess(\Gamma)$.  If not, then as above, the $2$-cells containing the
  portion of $Q$ immediately before and after the point which fails to
  be an immersion will form a cancellable pair in $D$.  The difference
  is that this time the two $2$-cells are not sent to $X$ in identical
  ways; they might be sent to redundant $2$-cells or in different ways
  to a single $2$-cell attached by a proper power.  Finally,
  $\phi(\psi(Q))$ is essential in $X$ since the inclusion
  $\Ess(\Gamma)\hookrightarrow X$ is $\pi_1$-injective by assumption.
  On the other hand, $D$ is simply-connected, so $Q$ is null-homotopic
  in $D$ and its image should be null-homotopic in $X$.  This
  contradiction shows that $\Gamma'$ is a forest.
\end{proof}

\begin{rem}[Structure of $\Gamma'$]
  The conclusion of Theorem~\ref{thm:forest} does not preclude the
  existence of trivial components in the interior of $D$ since the
  arguments given need an edge to get started.  Such isolated interior
  points can arise if $\Gamma$ passes through a $0$-cell of $X$.
  Another complication is that the components of $\Gamma'$ can be
  quite complicated trees.  Figure~\ref{fig:branching} illustrates how
  such branching can occur.  Although we will not need this
  simplification, we note that neither complication will occur when
  $\Gamma\cap X^0$ is empty.
\end{rem}

\begin{figure}[ht]
  \includegraphics[scale=1.3]{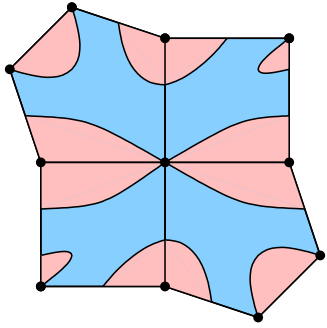}
  \caption{An example of branching in $\Gamma'$}
  \label{fig:branching}
\end{figure}

Despite the fact that $\Gamma'$ might branch in $D$, there is enough
structure to ensure that $D$ is constructed by gluing together
components in a tree-like fashion.  To make this precise we introduce
the idea of a connection graph.

\begin{defn}[Connection graph]\label{def:connection}
  If $D$ is a disc diagram and $\Gamma\hookrightarrow D$ is a graph in
  $D$, then we define its \emph{connection graph} $\Conn(\Gamma,D)$ as
  follows.  The vertices of $\Conn(\Gamma,D)$ are the path components
  of $\Gamma$ and the path components of $D\setminus \Gamma$, and we
  have an edge from $u$ to $v$ when $u$ represents a component
  $\Gamma_0$ of $\Gamma$, $v$ represents a component $D_0$ of
  $D\setminus \Gamma$, and $\Gamma_0 \cap \partial D_0 \neq
  \emptyset$.
\end{defn}

\begin{rem}[Paths]\label{rem:paths}
  Since the components involved are path connected and the edges
  represent adjacency in $D$, for any combinatorial path $P\to
  \Conn(\Gamma,D)$, we can create a path $Q\to D$ which traces through
  the corresponding components in the exact same order.  Moreover, if
  $P$ is simple, we can choose $Q$ to be simple.  Conversely, generic
  paths $Q\to D$ determine combinatorial paths $P\to\Conn(\Gamma,D)$
  which simply trace the components traversed. Our standing assumption
  that maps can be suitably subdivided to be combinatorial, rules out
  pathological paths which wiggle across a single edge in $\Gamma$
  infinitely often in a decaying manner.
\end{rem}

\begin{lem}[Tree-like]\label{lem:tree-like}
  Let $D$ be a disc diagram and let $\Gamma \hookrightarrow D$ be a
  forest in $D$.  If the leaves of $\Gamma$ lie in $\partial D$, then
  its connection graph, $T=\Conn(\Gamma,D)$, is a tree.  If in
  addition no isolated vertices of $\Gamma$ are contained in the
  interior of $D$, then the components of $D \setminus \Gamma$ are
  simply connnected.
\end{lem}

\begin{proof}
  Suppose $P\rightarrow T$ is a nontrivial closed simple cycle.  We
  will reach a contradiction by showing that $P$ has a {\em backtrack}
  meaning that it traverses an edge followed by its inverse.  Let
  $Q\rightarrow D$ be the closed simple cycle from
  Remark~\ref{rem:paths}.  Then $Q$ bounds a disc diagram $D'\subset
  D$. Since $P$ is nontrivial, $Q$ intersects $\Gamma$. An innermost
  component of $D'-\Gamma$ determines a backtrack of $P$.

  Suppose some component $D_0$ of $D\setminus \Gamma$ is not
  simply-connected.  Let $Q\subset D_0$ be an essential simple closed
  curve.  Let $D'$ be the region bounded by $Q$. Then $D'$ cannot be a
  disc since $Q$ is essential. Thus $D'$ contains some component of
  $\Gamma$, which is necessarily a trivial component since any
  nontrivial component intersects $\partial D$ by
  Theorem~\ref{thm:forest}.
\end{proof}

In order to take full advantage of Lemma~\ref{lem:tree-like} we
introduce the notion of a modified preimage.

\begin{defn}[Modified preimages]\label{def:preimages}
  Let $X$ be $2$-complex and let $Y$ and $Z$ be complementary
  subcomplexes separated by $\Gamma =Y\cap Z$.  If $\psi:D\to X$ is a
  disc diagram over $X$, then we partition $D$ into sets $Y'$, $Z'$
  and $\Gamma'$ as follows.  Let $\Gamma'$ be $\psi^{-1}(\Gamma)$ with
  any isolated points in the interior of $D$ removed and let $Z' =
  \psi^{-1}(Z) \setminus \Gamma'$ and $Y'=\psi^{-1}(Y) \setminus
  \Gamma'$. We will call $\Gamma'$, $Y'$ and $Z'$ the \emph{modified
    preimages} of $\Gamma$, $Y$ and $Z$, respectively.  Notice that
  $Y'$ and $Z'$ are open in $D$ and $\Gamma'$ is closed.
\end{defn}

The sets $Z'$ and $Y'$ are almost the same as $\psi^{-1}(Z \setminus
\Gamma)$ and $\psi^{-1}(Y \setminus \Gamma)$ except that the isolated
points of $\psi^{-1}(\Gamma)$ in the interior of $D$ have been added
to the regions which contain them.  Adding these points will ensure
that the components of $Z'$ and $Y'$ will be simply-connected whenever
$\Gamma'$ is a forest with all its leaves in $\partial D$.
In particular, the following corollary is an immediate consequence of
Theorem~\ref{thm:forest}, Lemma~\ref{lem:tree-like}, and
Definition~\ref{def:preimages}.

\begin{cor}[Simply-connected]\label{cor:connected}
  Let $X$ be a $2$-complex and let $\Gamma\hookrightarrow X$ be the
  subgraph which separates a splitting windmill $Z$ from its
  complement $Y$.  If $P\to X$ is a non-trivial null-homotopic
  immersed combinatorial path and $\psi:D\to X$ is a minimal area disc
  diagram having $P$ as its boundary cycle, then each component of
  each modified preimage, $\Gamma'$, $Y'$ and $Z'$ is
  simply-connected.
\end{cor}

A non-singular subdiagram of a disc diagram $D$ which is attached to
the rest of $D$ at a single point is a \emph{dangling subdiagram}.  As
a quick illustration of Lemma~\ref{lem:tree-like} we give a short
proof of the well-known result that certain disc diagrams must contain
dangling subdiagrams.

\begin{lem}[Dangling subdiagrams exist]\label{lem:dangling}
  If $X$ is a $2$-complex, $P\to X$ is a non-trivial null-homotopic
  immersed combinatorial loop, and $D\to X$ is a disc diagram having
  $P$ as its boundary cycle, then either $D$ itself is non-singular or
  $D$ contains at least two dangling subdiagrams.
\end{lem}

\begin{proof}
  Let $\Gamma$ be the collection of $0$-cells of $D$ whose removal
  disconnects $D$ (i.e. cut vertices) and note that a disc diagram
  without cut vertices is either trivial, a single $1$-cell, or
  non-singular.  By Lemma~\ref{lem:tree-like}, $T=\Conn(\Gamma,D)$ is
  a tree, and by construction each vertex of $\Gamma$ corresponds to a
  vertex of $T$ with valence at least $2$.  Thus the leaves of $T$
  correspond to components of $D\setminus \Gamma$ attached to the rest
  of $D$ at a single point.  Since trivial subdiagrams cannot be
  separated off by cut vertices and $1$-cells attached at a single
  point are prohibited since $P$ is immersed, the leaves of $T$
  correspond to dangling subdiagrams.  Similarly, if $T$ is trivial,
  then $D$ is non-singular since the restrictions on $P$ ensure that
  $D$ is not a single $0$-cell or single $1$-cell.  The result now
  follows from the observation that finite trees are either trivial or
  have at least two leaves.
\end{proof}

Our second application is only slightly more complicated.  In order to
state the result we will need the notion of an outermost component.

\begin{defn}[Outermost components]\label{def:outer}
  Let $Z$ be a subcomplex of a $2$-complex $X$, let $\psi:D\to X$ be a
  disc diagram, and let $Z'$ be the modified preimage of $Z$. A
  component $Z_0$ of $Z'$ is \emph{outermost} if $Z' \setminus Z_0$ is
  contained in a single connected component of $D \setminus Z_0$.
\end{defn}

\begin{lem}[Outermost components exist]\label{lem:outer}
  Let $Z$ be a splitting windmill in a $2$-complex $X$, let $P\to X$
  be a non-trivial null-homotopic immersed combinatorial path and let
  $\psi:D\to X$ be a minimal area disc diagram having $P$ as its
  boundary cycle.  If $D$ is non-singular and $Z'$ is the modified
  preimage of $Z$ in $D$, then either $Z'$ is connected or $Z'$ has at
  least two outermost components.  Moreover, for each outermost
  component $Z_0$ of $D$ there exists a simple path $Q\to D$ in
  $\partial Z_0$ so that $D\setminus Q$ is disconnected and $Z_0$ lies
  in a different connected component of $D\setminus Q$ from the rest
  of $Z'$.
\end{lem}

\begin{proof}
  By Theorem~\ref{thm:forest}, $\Gamma'$ is a forest with its leaves
  in $\partial D$, and so by Lemma~\ref{lem:tree-like}, the connection
  graph $T=\Conn(D,\Gamma')$ is a tree.  Consider the smallest subtree
  $T'$ of $T$ which contains all of the vertices corresponding to
  components of $Z'$.  This subtree is either trivial, in which case
  $Z'$ is connected, or it has at least two leaves.  By minimality of
  $T'$ each leaf corresponds to a component of $Z'$, and using
  Remark~\ref{rem:paths}, we see that a component of $Z'$ is an
  outermost component if and only if it corresponds to a leaf in $T'$.

  The final assertion can be shown as follows.  Let $Z_0$ be an
  outermost component which corresponds to a leaf $v$ in $T'$ and let
  $\Gamma_0$ be the component of $\Gamma'$ which corresponds to the
  unique vertex $u$ in $T'$ connected to $v$.  The intersection
  $\partial Z_0 \cap \Gamma_0$ will be a path with the required
  properties.  In particular, the intersection $\partial Z_0 \cap
  \Gamma_0$ is a simple path $Q$ (rather than more complicated
  $1$-complex) since $\Gamma_0$ is a tree with its leaves in $\partial
  D$ and $\partial Z_0$ is a circle, so $\partial Z_0\cap \Gamma_0$
  consists of at most one arc since $T$ is a tree.  The separation
  properties for $Q$ follow immediately from the position of $u$ in
  $T'$ and Remark~\ref{rem:paths}.
\end{proof}

Our third application of Lemma~\ref{lem:tree-like} will show that
certain disc diagrams contain $2$-cells which are extreme in the
following sense.

\begin{defn}[Extreme $2$-cells]\label{def:extreme}
  Let $X$ be a $2$-complex, let $Y$ and $Z$ be complementary
  subcomplexes, and let $\Gamma = Y\cap Z$ be the subgraph which
  separates them.  A $2$-cell $R$ in a disc diagram $\psi:D\to X$ is
  \emph{extreme with respect to $Z$} if $\partial R$ is the
  concatenation of two paths $S$ and $Q$ where $Q$ is a subpath of
  $\partial D$ and $S \cap \psi^{-1}(Z)$ has at most one non-trivial
  component (isolated points in the intersection are
  ignored). Figure~\ref{fig:extreme} contains a sketch of a disc
  diagram which contains four copies of the $2$-cell from
  Figure~\ref{fig:shading}.  The one in the lower lefthand corner is
  not extreme; the other three are extreme.
\end{defn}

\begin{figure}[ht]\centering
  \includegraphics[width=.6\textwidth]{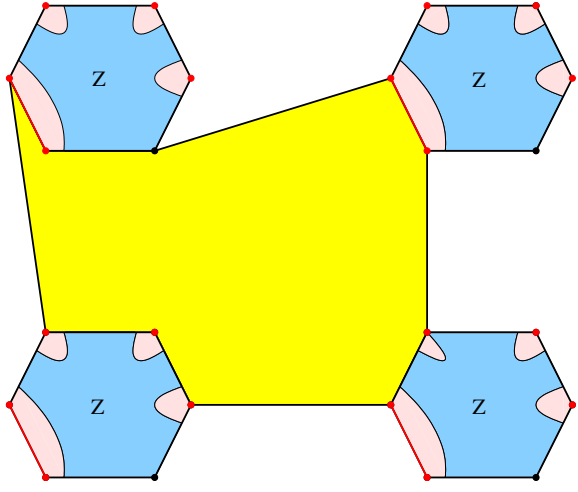}
  \caption{Three extreme $2$-cells in a disc
    diagram.}\label{fig:extreme}
\end{figure}

We will prove three versions of the following result under
successively weaker hypotheses.

\begin{lem}[Extreme $2$-cells exist: first version]\label{lem:extreme1}
  Let $X$ be a $2$-complex, let $Z$ be a splitting windmill in $X$
  with complement $Y$, let $P\to X$ be a non-trivial null-homotopic
  immersed combinatorial path, and let $\psi:D\to X$ be a minimal area
  disc diagram having $P$ as its boundary cycle.  If $D$ is
  non-singular, and the modified preimage of $Z$ in $D$ is connected,
  then either $D$ consists of a single $2$-cell or $D$ contains at
  least two $2$-cells which are extreme with respect to $Z$.
\end{lem}

\begin{proof}
  Let $Y'$ and $Z'$ denote the modified preimages of $Y$ and $Z$, and
  let $\Delta$ be the graph $D^1\cap Z'$.  Observe that $\Delta$ is a
  forest, for otherwise there would be a simple closed curve $Q$ in
  $\Delta$ and, since $D$ is simply-connected, $Q$ would bound a
  non-singular subdiagram of $D$ containing at least one $2$-cell.
  Consequently, $Q$ would lie in $Z'$ but contain points of $Y'$ in
  its interior, contradicting the fact that $Z'$ is simply-connected
  (Corollary~\ref{cor:connected}).  Moreover, the leaves of $\Delta$
  must lie in $\partial Z'$ because the $2$-cells of $X$ are attached
  along immersed paths (Definition~\ref{def:combinatorial}).  Thus, by
  Lemma~\ref{lem:tree-like}, the connection graph $T=\Conn(Z',\Delta)$
  is a tree.

  A similar argument shows that for each $2$-cell $R$ in $D$, the
  distinct portions of $\partial R \cap Z'$ (recall that there are at
  least two by the definition of a windmill) belong to distinct
  components of $\Delta$.  If not, a simple path in $\Delta$
  connecting distinct portions, combined with a simple path connecting
  them through $R\cap Z'$ (which exists because $R\cap Z'$ is
  connected) forms a simple closed path in $Z'$ which surrounds points
  in $Y' \cup \Gamma'$ (in particular there are points of this type in
  $\partial R$ separating the distinct intervals of $\partial R\cap
  Z'$ we have connected).  This contradicts that $Z'$ is
  simply-connected, proving the claim.  Consequently, all leaves of
  $T$ are components of $\Delta$.

  Finally, let $T'$ be the smallest subtree in $T$ which contains all
  of the vertices corresponding to components of $Z'\setminus \Delta$.
  Since the components of $Z'\setminus \Delta$ also correspond to the
  $2$-cells in $D$, $T'$ is a single vertex if and only if $D$
  consists of a single $2$-cell.  Moreover, when $D$ has more than one
  $2$-cell it is easy to see that a $2$-cell of $D$ is extreme with
  respect to $Z$ if and only if it corresponds to a leaf of $T'$.
\end{proof}

Using Lemma~\ref{lem:outer} we can remove the assumption that $Z'$ is
connected.

\begin{lem}[Extreme $2$-cells exist: second version]\label{lem:extreme2}
  Let $X$ be a $2$-complex, let $Z$ be a splitting windmill in $X$
  with complement $Y$, let $P\to X$ be a non-trivial null-homotopic
  immersed combinatorial path, and let $\psi:D\to X$ be a minimal area
  disc diagram having $P$ as its boundary cycle.  If $D$ is
  non-singular, then either $D$ consists of a single $2$-cell or $D$
  contains at least two $2$-cells which are extreme with respect to
  $Z$.
\end{lem}

\begin{proof}
  Let $Z'$ be the modified preimage of $Z$ in $D$.  By
  Lemma~\ref{lem:extreme1}, we may assume $Z'$ is disconnected and by
  Lemma~\ref{lem:outer}, $D$ must contain at least two outermost
  components.  If each outermost component $Z_0$ contributes at least
  one extreme $2$-cell, we will be done.

  Let $D_0$ be the union of the $2$-cells of $D$ which intersect $Z_0$
  non-trivially.  Notice that $D_0 \cap Z' = Z_0$ since the
  intersection of $Z'$ with each $2$-cell is connected.  We claim that
  $D_0$ is a nonsingular disc diagram which is attached to the rest of
  $D$ along a path $Q'$ contained in $Y'\cup \Gamma'$.  To see that
  $D_0$ is simply-connected, suppose not.  Then there is a simple
  closed path in $\partial D_0 \cap (Y'\cup \Gamma')$ which bounds a
  subdiagram of $D$ (it cannot contain points in $Z'$ since $Z'$ is
  open in $D$).  This subdiagram contains at least one $2$-cell and
  hence a point in $Z'$.  And finally the boundary of this component
  of $Z'$ must be an essential cycle in $\Gamma'$ contradicting
  Corollary~\ref{cor:connected}.  Thus $D_0$ is a disc diagram.  Since
  it is a union of $2$-cells and $Z_0$ is open in $D$, it is also
  non-singular.  Finally, by Lemma~\ref{lem:outer}, $Z_0$ can be
  separated from the rest of $Z'$ by a path $Q$ in $\Gamma'$.  Let
  $Q'$ be the portion of $\partial D_0$ which has the same endpoints
  as $Q$ and which avoids $Z_0$.  The path $Q'$ exists since $Q$
  separates and $Z'\cap D_0$ is connected.

  By Lemma~\ref{lem:extreme1} $D_0$ is either a single $2$-cell or it
  contains at least two $2$-cells which are extreme with respect to
  $Z$.  Since a single $2$-cell attached to the rest of $D$ along a
  path $Q'$ in $Y' \cup \Gamma'$ is always extreme with respect to
  $Z$, we may assume $D_0$ has at least two extreme $2$-cells.
  Finally, when such a $D_0$ is attached to the rest of the diagram
  along a path $Q'$ in $Y' \cup \Gamma'$, at most one of these
  $2$-cells loses its status as an extreme $2$-cell, and the proof is
  complete.
\end{proof}

Finally, using Lemma~\ref{lem:dangling} we can remove the assumption
that $D$ is non-singular.

\begin{thm}[Extreme $2$-cells exist]\label{thm:extreme}
  If $X$ is a $2$-complex, $Z$ is a splitting windmill in $X$ with
  complement $Y$, $P\to X$ is a non-trivial null-homotopic immersed
  combinatorial path, and $\psi:D\to X$ is a minimal area disc diagram
  having $P$ as its boundary cycle, then either $D$ consists of a
  single $2$-cell or $D$ contains at least two $2$-cells which are
  extreme with respect to $Z$.
\end{thm}

\begin{proof}
  We may assume that $D$ is singular by Lemma~\ref{lem:extreme2}, so
  $D$ must contain at least two dangling subdiagrams by
  Lemma~\ref{lem:dangling}.  If each dangling subdiagram $D'$
  contributes at least one extreme $2$-cell, we will be done.  By
  Lemma~\ref{lem:extreme2} $D'$ is either a single $2$-cell or it
  contains at least two $2$-cells which are extreme with respect to
  $Z$.  Since a single $2$-cell attached to the rest of $D$ at a point
  is always extreme with respect to $Z$, we may assume $D'$ has at
  least two extreme $2$-cells.  Finally, when such a $D'$ is attached
  to the rest of the diagram at a point, at most one of these
  $2$-cells loses its status as an extreme $2$-cell, and the proof is
  complete.
\end{proof}

When the hypotheses of Theorem~\ref{thm:extreme} hold, we will say
that \emph{disc diagrams over $X$ have extreme $2$-cells}.

\section{Applications to coherence}\label{sec:apps}

In this final section we combine the constructions $\partial A$ and
$\eth A$ with Theorem~\ref{thm:extreme} to show that various groups
are coherent.  Throughout this section let $X$ be a $2$-complex, let
$A$ be a portion of its $1$-skeleton, let $\Gamma$ be either $\partial
A$ or $\eth A$, and let $Y$ and $Z$ be as defined in
Definitions~\ref{def:pA} or~\ref{def:eA}, respectively.  In order to
apply Theorem~\ref{thm:extreme}, we need to know that $Z$ is a
splitting windmill.  As we noted in the definitions of $\partial A$
and $\eth A$, there are easy conditions on $A$ which ensure that $Z$
is a windmill, so the main issue becomes whether $\Ess(\Gamma)\to X$
is $\pi_1$-injective.  Moreover, since the inclusion map
$\Ess(\Gamma)\to X$ can be homotoped to a map $\Ess(\Gamma)\to A
\subseteq X$ by pushing the regular neighborhood of $A$ back into $A$
in the obvious way, it is sufficient to establish that this new map is
$\pi_1$-injective and that the inclusion $A\to X$ is
$\pi_1$-injective.  Here are three common situations where $A\to X$ is
known to be $\pi_1$-injective.

\begin{thm}[Freiheitsatz]\label{thm:frei}
  Let $X$ be the standard $2$-complex of a presentation whose single
  relator is reduced and cyclically reduced.  If $A$ is a non-empty
  subgraph of $X^1$ that omits at least one $1$-cell contained in the
  boundary cycle of the relator, then the inclusion $A\hookrightarrow
  X$ is $\pi_1$-injective.
\end{thm}

The Freiheitssatz for one-relator groups, which was first proven by
Magnus, can be generalized in various ways. One of these
generalizations involves the notion of staggered $2$-complex (see
\cite{LySch77} or \cite{HruskaWise-Torsion}).

\begin{defn}[Staggered]
  Let $X$ be a $2$-complex with a subgraph $A\subset X^1$ such that
  each $2$-cell of $X$ contains a $1$-cell not in $A$ on its
  boundary. Suppose that there is a linear ordering on the $1$-cells
  of $X$ which are not in $A$, and a linear ordering on the $2$-cells
  of $X$. For each $2$-cell $\alpha$, we let $\max(\alpha)$ and
  $\min(\alpha)$ denote the highest and lowest $1$-cells not in $A$
  which occur in $\partial \alpha$. We then say that the pair $X,A$ is
  \emph{staggered} provided that if $\alpha$ and $\beta$ are $2$-cells
  with $\alpha < \beta$ then $\max(\alpha) < \max(\beta)$ and
  $\min(\alpha) < \min(\beta)$.
\end{defn}

The following generalization of the Freiheitssatz is proven in
\cite{LySch77} (see also \cite{HruskaWise-Torsion}).

\begin{thm}\label{thm:staggered}
  If $X,A$ is staggered (for some linear orderings) then the inclusion
  map $A\hookrightarrow X$ is $\pi_1$-injective on all components.
\end{thm}

Our third example is an immediate corollary of the fundamental theorem
of small cancellation theory.  See \cite{LySch77} or \cite{McWi-fl}
for small-cancellation definitions and further details.

\begin{thm}\label{thm:small} 
  Let $X$ be a $C(6)$ \ $[C(4)-T(4)]$ small-cancellation complex, and
  let $A$ be a subgraph of $X^1$.  If there does not exist a path $S$
  in $A$ and a path $Q$ in $X$ such that $Q$ is the concatenation of
  at most $3$ pieces $[$$2$ pieces$]$ in $X$ and $QS$ is the attaching
  map of a $2$-cell of $X$, then $A\hookrightarrow X$ is
  $\pi_1$-injective.
\end{thm}

Thus in each of these three contexts we merely need to check that
$\Ess(\Gamma)\to A$ is $\pi_1$-injective in order for
Theorem~\ref{thm:extreme} to apply.

\subsection{Combinatorial descriptions}

To understand the situation, we now provide a combinatorial
description of $\partial A\rightarrow A$ and $\eth A\to A$.

\begin{defn}\label{def:combinatorial injectivity}
  Let $X$ be the standard $2$-complex of the presentation
  \[\langle a_1,\ldots,a_p, b_1,\ldots, b_q \mid W_1,\ldots,
  W_r\rangle\] 
  Let $A$ and $B$ be the subgraphs of $X^1$ corresponding to the $a_i$
  and $b_i$ edges.  For each $i$, the word $W_i$ can be written
  uniquely in the form
  \[W_{i0}b_{i1}^{\epsilon_{i1}}W_{i1}b_{i2}^{\epsilon_{i2}} W_{i2}
  \ldots b_{is_i}^{\epsilon_{is_i}}W_{is_i}\]
  where each $\epsilon_{ij}$ is $\pm 1$, each $b_{ij}$ is a generator
  in $B$, and each word $W_{ij}$ is a (possibly empty) word in the
  generators of $A$.  By replacing $W_i$ with one of its cyclic
  conjugates we can assume that $W_{i0}$ is empty.  We now form a
  graph $\partial A$ from the set of $W_i$ words as follows: For each
  $i$ we form a $2s_i$-sided polygon whose edges are directed and
  labeled by the elements $b_{ij}^{\epsilon_{ij}}$ and $W_{ij}$ in
  exactly the same order as in $W_i$.  For each $k$ we identify edges
  which are labeled by $b_k$ according to their orientations.  Finally
  for each $k$, we remove the interior of the edge labeled~$b_k$.  The
  resulting graph $\partial A$ has $2q$ vertices and $\sum_{i=1}^r
  s_i$ edges.

  By assumption, each word $W_{ij}$ is a word in the free group
  generated by $A$, and there is an induced label-preserving map from
  $\partial A$ to $A$.  Note that the edges which are labeled by the
  trivial element are mapped to vertices.  The graph $\partial A$ is
  {\it injective} if this map is $\pi_1$-injective on each component.
  An important special case where $\partial A$ is injective is when
  the words $W_{ij}$ form a basis for a subgroup of the free group
  generated by $A$.
\end{defn}

\begin{defn}[Generator Graphs]\label{def:injective}
  Let $W$ be an arbitrary word and let $t$ be one of the generators it
  contains.  If we single out all of the instances of $t$ in $W$ then
  we can write $W$ uniquely in the form $W_0t^{\epsilon_1}W_1
  t^{\epsilon_2} W_2 \ldots t^{\epsilon_r}W_r$ where each $\epsilon_i$
  is an integer and each word $W_i$ is a non-empty word which does not
  contain the letter $t$.  If we replace $W$ with one of its cyclic
  conjugates we can assume that $W_0$ is empty.  We now form a graph
  $\eth t(W)$, called the {\it generator graph} of $W$ for the
  generator $t$, as follows: We begin with the $|W|$-sided polygon
  whose edges are directed and labeled by the generators so that the
  label of the entire boundary is the word $W$.  Next we identify all
  of the $t$-edges according to their orientations, and finally we
  remove the interior of the unique edge labeled~$t$ in the quotient.
  The resulting graph will be $\eth t(W)$.  Notice that it contains
  either one or two connected components.

  More generally, let $B$ and $C$ be disjoint sets of letters and let
  $W$ be a word of the form $W = B_1C_1B_2C_2\ldots B_kC_k$ where
  $B_i$ and $C_i$ are non-empty reduced words using generators from
  $B$ and $C$ respectively.  The generator graph $\eth B(W)$ is formed
  as follows: Take the $\size{W}$-sided polygon as before, and
  identify all of the instances of the generator $b \in B$ according
  to their orientations, and repeat this for each generator in $B$
  that occurs in $W$.  Finally, remove the interior of the edges
  labeled by elements of $B$.  The resulting graph is $\eth B(W)$.
  This more general graph may contain quite a few components.

  Since each $C_i$ is a word in the free group generated by $C$, there
  is an induced label-preserving map from $\eth B(W)$ to the bouquet
  of circles labeled by the $c_i \in C$.  The graph $\eth B(W)$ is
  {\it injective} if this map is $\pi_1$-injective on each component.
  An important special case where $\eth B(W)$ is injective is when the
  words $C_i$ form a basis for a subgroup of the free group generated
  by $C$.
\end{defn}

\subsection{Applications of $\partial A$}
Here is an application to coherence of one-relator groups.

\begin{thm}\label{thm:1rel-extreme}
  Consider a one-relator group of the form
  $$G=\langle a_1, a_2,\dots, b \mid \big( b^{\epsilon_1}W_1 b^{\epsilon_2} W_2
  \ldots b^{\epsilon_r}W_r \big)^n \rangle $$ where for each $i$,
  $\epsilon_i =\pm1$, $n$ is arbitrary, and $W_i$ is a word in the
  ${a_i}$.  Suppose that $\partial A$ is injective.  Let $P$ denote a
  reduced word representing the trivial element, then $P$ contains a
  subword $Q$ such that $QS$ is equal to a cyclic conjugate of $W^{\pm
    n}$ and $b^{\pm 1}$ occurs at most once in $S$.  As a consequence,
  $G$ is coherent.
\end{thm}

\begin{proof}
  By the Freiheitssatz (Theorem~\ref{thm:frei}), the $a_i$ elements
  form a basis for a free group.  Let $R$ denote the unique $2$-cell
  of $X$.  Let each side of $R$ at $b$ have weight $1$, and let each
  side of $R$ not at $b$ have weight $0$.  Then $X$ satisfies the
  $\leq$ condition for the perimeter reduction hypothesis of
  \cite[Thm~7.6]{McCammondWiseCoherence}, and is therefore coherent.
\end{proof}

We can now state a generalization of Theorem~\ref{thm:1rel-extreme} to
staggered $2$-complexes.

\begin{thm}\label{thm:relatively-staggered extreme cell}
  Let $X$ be the standard $2$-complex of the presentation \[\langle
  a_1,\ldots, a_p, t_1,\ldots, t_q \mid W_1,\ldots, W_r\rangle\] and
  let $A$ denote the subgraph of $X^1$ corresponding to the $a_i$
  edges. If $X,A$ is relatively staggered for some linear orderings
  and the inclusion $\eth A\to X$ is injective, then extreme $2$-cells
  exist in disc diagrams over $X$.
\end{thm}

\begin{thm}\label{thm:smallcancellation analog} 
  Let $X$ be a $C(6) [C(4)-T(4)]$ small-cancellation complex, and
  suppose that $A$ is a subgraph of $X^1$ such that there does not
  exist a path $S\rightarrow A$ such that $QS$ is the attaching map of
  a $2$-cell of $X$, where $Q$ is the concatenation of at most $3$
  pieces $[$$2$ pieces] in $X$.  Then $A\rightarrow X$ is
  $\pi_1$-injective.  Consequently, if $\partial A\to A$ is injective,
  then extreme $2$-cells exist in disc diagrams over $X$.
\end{thm}

In both cases, the restricted nature of the extreme $2$-cells,
combined with Theorem~7.6 of \cite{McCammondWiseCoherence}, leads to
new tests for coherence.

\subsection{$\eth A$ applications}

The following theorem is merely the conclusion of
Theorem~\ref{thm:extreme} translated into a more group theoretic
language.

\begin{thm}\label{thm:extreme-cell}
  Let $W = A_1 B_1 A_2 B_2 \ldots A_k B_k$ be a word where the $A_i$
  are non-empty words using generators in $A$ and the $B_i$ are
  non-empty words using generators from $B$ (disjoint from $A$), and
  let $G$ be the one-relator group $G = \langle A \cup B\mid W^n
  \rangle$.  If $\Ess(\eth A)\to A$ is $\pi_1$-injective and $P$ is a
  cyclically reduced word representing the trivial element in $G$,
  then there are words $Q$ and $S$ such that $Q$ is a subword of $P$,
  $QS$ is a cyclic conjugate of $W^{\pm n}$ and $S$ is a subword of
  $B_{i-1}A_iB_i$ for some $i$ where the subscripts are considered
  $\mod k$.
\end{thm}

When $k$ is at least~$2$ then this theorem gives a refinement of the
B.B.~Newman spelling theorem in the sense that it further restricts
the size of the possible complements $S$.  As with the spelling
theorem, this leads immediately to a corresponding weight test.  We
refer the reader to \cite{McCammondWiseCoherence} for the definition
of $\perimeter(A_i)$ and $\weight(W^n)$.

\begin{cor}\label{cor:AB-case}
  Let $G = \langle A \cup B\mid W^n \rangle$ be a one relator group
  with torsion where $A$ and $B$ are disjoint sets of generators and
  $W$ has the form $A_1 B_1 A_2 B_2 \ldots A_k B_k$ for some non-empty
  words $A_i$ and $B_i$ using generators from $A$ and $B$
  respectively.  If $\Ess(\eth A)\to A$ is $\pi_1$-injective and
  $\perimeter(A_i) \leq \weight(W^n)$ for all $i$, then $G$ is
  coherent.
\end{cor}

\begin{proof}
  Let $X$ be the standard $2$-complex of the presentation.  Assign a
  weight of $1$ to each side labeled by an element of $A$ and a weight
  of $0$ to each side labeled by an element of $B$.  By
  Theorem~\ref{thm:extreme-cell}, the Perimeter Reduction Hypothesis
  of \cite[Thm~7.6]{McCammondWiseCoherence} is satisfied and so
  $G\cong\pi_1X$ is coherent.
\end{proof}

The most important Corollary, and the easiest to apply, is the
following.

\begin{cor}\label{cor:t-windmills}
  Let $G = \langle A,t\mid W^N \rangle$ where $W$ has the form
  $t^{\epsilon_1}W_1 t^{\epsilon_2} W_2 \ldots t^{\epsilon_k}W_k$ and
  for each $i$, $\epsilon_i$ is an integer and $W_i$ is a reduced word
  over $A$.  If $\Ess(\eth A)\to A$ is injective, then $G$ is
  coherent.
\end{cor}

\begin{proof}
  Since $\perimeter(t)=\weight(W)$, Corollary~\ref{cor:AB-case}
  applies.
\end{proof}

\bibliographystyle{plain}
\def\cprime{$'$}

\end{document}